\documentclass[leqno,12pt]{article} \textheight 8.8 true in \textwidth
6.2 true in

\usepackage[T2A]{fontenc}
\usepackage[cp1251]{inputenc}
\usepackage[english]{babel}
\usepackage{amsfonts}
\usepackage{amssymb,amsmath,amstext}
\usepackage{amscd, amsthm, latexsym}
\usepackage{color}
\usepackage[final]{graphicx}
\usepackage{epsfig}
\usepackage{euscript}
\usepackage{srcltx,epsf}
\usepackage{wrapfig}
\usepackage{mathrsfs}
\usepackage{ifthen}
\usepackage{cleveref}
\usepackage{multirow}
\usepackage{array}
\usepackage{diagbox}
\usepackage{rotating}

\def\hide#1{}
\def\old#1{}

\def\oop#1{}
\def\gap#1{}

\theoremstyle{plain}
\newtheorem{theorem}{Theorem}
\newtheorem{proposition}{Proposition}
\newtheorem*{theorem*}{Theorem}
\newtheorem{lemma}{Lemma}

\newtheorem{corollary}{Corollary}
\newtheorem{definition}{Definition}

\def \ZZ  {\Bbb Z}

\makeatletter

\def\address#1#2{\begingroup
\noindent\parbox[t]{7.8cm}{%
\small{\scshape\ignorespaces#1}\par\vskip1ex
\noindent\small{\itshape E-mail address}%
\/: #2\par\vskip4ex}\hfill%
\endgroup}%

\makeatother

\title{On the coverings of Hantzsche-Wendt manifold}
\author{
Grigory~Chelnokov
\;and
Alexander~Mednykh\thanks{The study was carried out within the
framework of the state contract of the
Sobolev Institute of Mathematics (project no. 0314-2019-0007)}\\
}
\date{}

\begin{document}

\newcounter{figcounter}
\setcounter{figcounter}{0} \addtocounter{figcounter}{1}

\maketitle 
\begin{abstract}
There are only 10 Euclidean forms, that is  flat closed three
dimensional manifolds: six are orientable
$\mathcal{G}_1,\dots,\mathcal{G}_6$ and four are non-orientable
$\mathcal{B}_1,\dots,\mathcal{B}_4$. In the present paper we
investigate the manifold $\mathcal{G}_6$, also known as
Hantzsche-Wendt manifold; this is the unique Euclidean $3$-form with
finite first homology group $H_1(\mathcal{G}_6) = \mathbb{Z}^2_4$.

The aim of this paper is to describe all types of $n$-fold coverings
over $\mathcal{G}_{6}$  and calculate the numbers of non-equivalent
coverings of each type. We classify subgroups in the fundamental
group $\pi_1(\mathcal{G}_{6})$ up to isomorphism.  Given index $n$,
we calculate the numbers of subgroups and the numbers of conjugacy
classes of subgroups for each isomorphism type and provide the
Dirichlet generating series for the above sequences.


\end{abstract}

\footnote{ 
2010 \textit{Mathematics Subject Classification}. Primary 20H15;
Secondary 57M10, 05A15, 55R10.}
\footnote{ 
\textit{Key words and phrases}. Euclidean form, platycosm, flat
3-manifold, non-equivalent coverings, crystallographic group,
Dirichlet generating series, number of subgroups, number of
conjugacy classes of subgroups.}


\section*{Introduction}
Let  $\mathcal{M}$  be a connected manifold with fundamental group
$G=\pi_{1}(\mathcal{M}).$ Two coverings
$$p_1: \mathcal{M}_1 \to \mathcal{M}   \text {      and       } p_2:
\mathcal{M}_2 \to \mathcal{M}  $$ are said to be equivalent if there
exists a homeomorphism $h:  \mathcal{M}_1 \to \mathcal{M}_2$ such
that $p_1 =p_2 \circ h.$ According to the general theory of covering
spaces, any $n$-fold covering is uniquely determined by a subgroup
of index $n$ in the group $G$. The equivalence classes of $n$-fold
coverings of $\mathcal{M}$ are in one-to-one correspondence with the
conjugacy
 classes of subgroups of index $n$ in the fundamental group
 $\pi_1(\mathcal{M}).$ See,
for example, (\cite{Hatch}, p.~67).
  In such a way the following natural problems arise: to
  describe the isomorphism classes of subgroups of finite index in
  the fundamental group of a given manifold and to enumerate the
  finite index subgroups and their conjugacy classes with respect to
  isomorphism type.

We use the following notations: let $s_G(n)$ denote the number of
subgroups of index $n$ in the group $G$, and let $c_G(n)$ be the
number of conjugacy classes of such subgroups. Similarly, by
$s_{H,G}(n)$ denote the number of subgroups of index $n$ in the
group $G$, which are isomorphic to $H$, and by $c_{H,G}(n)$ the
number of conjugacy classes of such subgroups. So, $c_G(n)$
coincides with the number of nonequivalent $n$-fold coverings over a
manifold $\mathcal{M}$ with fundamental group
$\pi_1(\mathcal{M})\cong G$, and $c_{H,G}(n)$ coincides with the
number of nonequivalent $n$-fold coverings $p: \mathcal{N}\to
\mathcal{M}$, where $\pi_1(\mathcal{N})\cong H$ and
$\pi_1(\mathcal{M})\cong G$. The numbers $s_G(n)$ and $c_G(n)$,
where $G$ is the fundamental group of closed orientable or
non-orientable surface, were found  in (\cite{Med78}, \cite{Med79},
\cite{MP86}). In the paper \cite{Medn}, a general method for
calculating the number $c_G(n)$  of conjugacy classes of subgroups
in an arbitrary finitely generated group $G$ was given. Asymptotic
formulas for $s_G(n)$ in many important cases were obtained in
\cite{Lub}.

The values of $s_G(n)$ for the wide class of 3-dimensional Seifert
manifolds were calculated in \cite{LisMed00} and \cite{LisMed12}.
The present paper is a part of the series of our papers devoted to
enumeration of finite-sheeted coverings  over closed Euclidean
3-manifolds. These manifolds are also known as flat 3-dimensional
manifolds or Euclidean 3-forms.

The class of such manifolds is closely related to the notion of
Bieberbach group. Recall that a subgroup of isometries of
$\mathbb{R}^3$ is called {\em Bieberbach group} if it is discrete,
cocompact and torsion free.  Each  $3$-form can be represented as a
quotient $\mathbb{R}^3/G$, where $G$ is a Bieberbach group. In this
case, $G$ is isomorphic to the fundamental group of the manifold,
that is $G\cong\pi_1(\mathbb{R}^3/G)$. Classification of three
dimensional Euclidean forms up to homeomorphism was  obtained  by W.
Nowacki \cite{Now}  and  W. Hantzsche and H. Wendt \cite{H-W}. There
are only 10 Euclidean forms: six are orientable
$\mathcal{G}_1,\dots,\mathcal{G}_6$ and four are non-orientable
$\mathcal{B}_1,\dots,\mathcal{B}_4$ See monograph \cite{Wolf} for
more details.


In our previous paper \cite{We1} we describe isomorphism types of
finite index subgroups $H$ in the fundamental group $G$ of manifolds
$\mathcal{B}_1$ and $\mathcal{B}_2$. Further, we calculate the
respective numbers $s_{H,G}(n)$ and $c_{H,G}(n)$ for each
isomorphism type $H$.
In subsequent articles \cite{We2}, \cite{We3} and \cite{We4} similar
questions were solved for manifolds $\mathcal{G}_2$,
$\mathcal{G}_3$, $\mathcal{G}_4$, $\mathcal{G}_5$, $\mathcal{B}_3$
and $\mathcal{B}_4$.

The aim of the present paper is to solve the same questions for the
Hantzsche--Wendt manifold $\mathcal{G}_6,$  undoubtedly the most
weird among Euclidean $3$-manifolds.  This is the unique Euclidean
$3$-form with finite first homology group $H_1(\mathcal{G}_6) =
\mathbb{Z}^2_4.$
Its fundamental set consists of two cubes described below. See  also
(\cite{Eve}, Table 4) for  description of a cubical fundamental
domain for $\mathcal{G}_6$. The Hantzsche--Wendt  manifold is   also
known as Fibonacci manifold $M_3.$  The Fibonacci manifold $M_n,\,
n\ge2$
 is a closed orientable three-dimensional manifold whose fundamental group is the
 Fibonacci group $F(2,2n)=\langle x_1,\ldots,x_{2n}: x_ix_{i+1}=x_{i+2},\,i\mod{2n}\rangle.$
These manifolds were discovered by H. Helling, A.C. Kim and J.
Mennicke \cite{HKM}.
 It was shown   by H.M. Hilden, M.T. Lozano and J.M. Montesinos \cite{a3} that
$ M_n$  is the $n$-fold cyclic covering of the three-dimensional
sphere $\mathbb{S}^3$ branched over the figure-eight knot.

 Also, by A.Yu. Vesnin and A.D. Mednykh \cite{a9}, $ M_3$  is the two fold covering of   $\mathbb{S}^3$ branched over the Borromean rings.
 The outer automorphism group of the Hantzsche--Wendt manifold  was calculated by B. Zimmermann \cite{Zim}. Its high  dimensional
 analogues were investigated by A. Szczepa{\'n}ski \cite{Szcz}.

The description of $\mathcal{G}_6$ through  Bieberbach group is the
following: the group is generated by isometries
\begin{align*}
S_1: (x,y,z) &\mapsto (x+1, -y, -z+1),\\
S_2: (x,y,z) &\mapsto (-x+1,y+1, -z),\\
S_3: (x,y,z) &\mapsto (-x, -y+1, z+1).
\end{align*}

In more geometric terms $\mathcal{G}_6$ can be described in the
following way. We take the union of two cubes $[-1,0]^3\cup [0,1]^3$
as the fundamental domain of $\mathcal{G}_6$, we call this cubes
positive and negative respectively. Now we have to provide six
isometries to align each face of the negative cube with the
respective face of the positive cube, we glue faces by this
alignment.
\begin{itemize}
\item we align faces $z=0$ and $z=-1$ with faces $z=1$ and $z=0$
through $(x,y,z)\mapsto (x+1,-y,-z+1)$ and $(x,y,z)\mapsto
(x+1,-y,-z-1)$ respectively (these are isometries $S_1$ and
$S_3^{-2}S_1$ respectively);
\item align faces $x=0$ and $x=-1$ with faces $x=1$ and $x=0$
through $(x,y,z)\mapsto (-x+1,y+1,-z)$ and $(x,y,z)\mapsto
(-x-1,y+1,-z)$ respectively (these are isometries $S_2$ and
$S_1^{-2}S_2$ respectively);
\item finally, align faces $y=0$ and $y=-1$ with faces $z=1$ and $z=0$
through $(x,y,z)\mapsto (-x,-y+1,z+1)$ and $(x,y,z)\mapsto
(-x,-y-1,z+1)$ respectively (these are isometries $S_3$ and
$S_2^{-2}S_3$ respectively).
\end{itemize}

In the present paper, we classify finite index subgroups in the
fundamental group $\pi_1(\mathcal{G}_{6})$ up to isomorphism. Given
index $n$, we calculate the numbers of subgroups and the numbers of
conjugacy classes of subgroups for each isomorphism type.  Also, we
provide the Dirichlet generating functions for all the above
sequences.

Numerical methods to solve these and similar problems for the
three-dimensional crystallographic groups were developed by the
Bilbao group \cite{babaika}. The convenience of language of
Dirichlet generating series for this kind of problems was
demonstrated in \cite{Ruth}. The first homologies of all the
three-dimensional crystallographic groups are determined in
\cite{Ratc}.

\subsection*{Notations}
Let $G$ be a group, $u$, $v$ are elements and $H$, $F$ are subgroups
in $G$. We use $u^v$ instead of $vuv^{-1}$ and $[u,v]$ instead of
$uvu^{-1}v^{-1}$ for the sake of brevity. By $H^v$ denote the
subgroup $\{u^v|\,u\in H\}$. By $H^F$ denote the family of subgroups
$H^v,\, v \in F$. By $Ad_v: G \to G$ denote the automorphism given
by $u \to u^v$.

By $s_{H,G}(n)$ we denote the number of subgroups of index $n$ in
the group $G$ isomorphic to the group $H$; by $c_{H,G}(n)$ the
number of conjugacy classes of subgroups of index $n$ in the group
$G$ isomorphic to the group $H$. Through this paper usually $G$ and
$H$ are fundamental groups of manifolds $\mathcal{G}_{i}$, in this
case we omit $\pi_1$ in indices.

 Also we will need the
following number-theoretic functions. Given a fixed $n$ we widely
use summation over all representations of $n$ as a product of two or
three positive integer factors $\textstyle \sum_{ab=n}$ and
$\textstyle\sum_{abc=n}$. The order of factors is important. We
assume this sum vanishes if $n$ is not integer.

To start with, this is the natural language to express the function
$\sigma_0(n)$ -- the number of representations of number $n$ as a
product of two factors $ \textstyle\sigma_0(n)=\sum_{ab=n}1. $
We will also need the following generalizations of $\sigma_0$:
\begin{align*}
\sigma_1(n)&=\sum_{ab=n}a, &
\sigma_2(n)&=\sum_{ab=n}\sigma_1(a)=\sum_{abc=n}a,\\
d_3(n)&=\sum_{ab=n}\sigma_0(a)=\sum_{abc=n}1, & \omega(n)
&=\sum_{ab=n}a\sigma_1(a)= \sum_{abc=n}a^2b.
\end{align*}

\section{Formulation of main results}
The main goal of this paper is to prove the following two theorems.

\begin{theorem}\label{th-1-didicosm}
Every subgroup $\Delta$ of finite index $n$ in
$\pi_{1}(\mathcal{G}_{6})$ is isomorphic to either
$\pi_{1}(\mathcal{G}_{6})$, or $\pi_{1}(\mathcal{G}_{2})$, or
$\ZZ^3$. The respective numbers of subgroups are
$$
s_{\mathcal{G}_1,\mathcal{G}_6}(n)=\omega\left(\frac{n}{4}\right),\leqno
(i)
$$
$$
s_{\mathcal{G}_2,\mathcal{G}_6}(n)=3\omega\left(\frac{n}{2}\right)-3\omega\left(\frac{n}{4}\right),\leqno
(ii)
$$
$$
s_{\mathcal{G}_6,\mathcal{G}_6}(n)=n\left(d_3(n)-3d_3\left(\frac{n}{2}\right)+3d_3\left(\frac{n}{4}\right)-d_3\left(\frac{n}{8}\right)\right).
\leqno (iii)
$$
\end{theorem}

\begin{theorem}\label{th-2-didicosm}
Let $\mathcal{N} \to \mathcal{G}_{6}$ be an $n$-fold covering over
$\mathcal{G}_{6}$. Then  $\mathcal{N}$ is homeomorphic to one of
$\mathcal{G}_{6}$, $\mathcal{G}_{2}$ or $\mathcal{G}_{1}$. The
corresponding numbers of nonequivalent coverings are given by the
following formulas:
$$
c_{\mathcal{G}_1,\mathcal{G}_6}(n)=\frac{1}{4}\omega\left({\frac{n}{4}}\right)+\frac{3}{4}\sigma_2\left(\frac{n}{4}\right)+\frac{9}{4}\sigma_2\left(\frac{n}{8}\right),
\leqno (i)
$$
$$
c_{\mathcal{G}_2,\mathcal{G}_6}(n)=\frac{3}{2}\left(\sigma_2\!\!\left(\frac{n}{2}\right)+2\sigma_2\!\!\left(\frac{n}{4}\right)-3\sigma_2\!\!\left(\frac{n}{8}\right)+d_3\!\!\left(\frac{n}{2}\right)-d_3\!\!\left(\frac{n}{4}\right)-3d_3\!\!\left(\frac{n}{8}\right)+5d_3\!\!\left(\frac{n}{16}\right)-2d_3\!\!\left(\frac{n}{32}\right)
\right), \leqno (ii)
$$
$$
c_{\mathcal{G}_6,\mathcal{G}_6}(n)=d_3(n)-3d_3\left(\frac{n}{2}\right)+3d_3\left(\frac{n}{4}\right)-d_3\left(\frac{n}{8}\right).
\leqno (iii)
$$
\end{theorem}

{\bf Remark.} If $n$ is odd then $\mathcal{N}\cong \mathcal{G}_{6}$.
If $n \equiv 2 \mod 4$ then $\mathcal{N}\cong \mathcal{G}_{2}$.
Finally, if $4 \mid n$ then $\mathcal{N}\cong \mathcal{G}_{2}$ or
$\mathcal{N}\cong \mathcal{G}_{1}$.

Dirichlet generating series for the sequences provided by Theorems 1
and 2 are given in Table~2 in Appendix.


\section{Preliminaries}
In this section we have collected some known statements that will be
used later.

\begin{proposition}\label{number of sublattices}
\begin{itemize}
\item[(i)] The sublattices of index $n$ in the $2$-dimensional
lattice $\ZZ^2$ are in one-to-one correspondence with the matrices
$\bigl(\begin{smallmatrix} b & c \\
 0 & a\end{smallmatrix}\bigr)$, where $a,b>0,\,ab=n$, $0 \le c < b$. Consequently,
the number of such sublattices is $\sigma_1(n)$.

\item[(ii)] The sublattices of index $n$ in the $3$-dimensional
lattice $\ZZ^3$ are in one-to-one correspondence with the integer matrices $\bigl(\begin{smallmatrix} c & e & f \\
0 & b & d \\ 0 & 0 & a\end{smallmatrix}\bigr)$, where $a,b,c >0,
\,abc=n$, $0 \le d < b$ and $0 \le f,e < c$. Consequently, the
number of such sublattices is $\omega(n)$.
\end{itemize}
\end{proposition}
For the proof see, for example, (\cite{We2}, Proposition 1).

\begin{corollary}\label{number of mirror-preserved_in_Z^2}
Let $\;\ell: \ZZ^2 \mapsto \ZZ^2$ be an automorphism of $\ZZ^2$,
given by $\ell(u,v)=(u,-v)$. The sublattices $\Delta$ of index $n$
in the $2$-dimensional lattice $\ZZ^2$ such that
$\ell(\Delta)=\Delta$ are in one-to-one correspondence with the
matrices in the union of the two families of integer matrices,
$\bigl(\begin{smallmatrix} b & 0 \\
 0 & a\end{smallmatrix}\bigr)$, where $a,b>0,\,ab=n$, and $\bigl(\begin{smallmatrix} b & a/2 \\
 0 & a \end{smallmatrix}\bigr)$, where $a,b>0,\,ab=n$ and $a$ is even. Consequently,
the number of such sublattices is
$\sigma_0(n)+\sigma_0\left(\frac{n}{2}\right)$.
\end{corollary}

\begin{corollary}\label{number of mirror-preserved}
Let $\;\ell: \ZZ^3 \mapsto \ZZ^3$ be an automorphism of $\ZZ^3$,
given by $\ell(u,v,w)=(u,v,-w)$. Then the number of subgroups
$\Delta$ of index $n$ in $\ZZ^3$ such that $\ell(\Delta)=\Delta$ is
equal to $\sigma_2(n)+3\sigma_2\left(\frac{n}{2}\right)$.
\end{corollary}

Proof in (\cite{We4}, Corollary 3).

In the next two propositions we  enumerate the subgroups $\Delta$ of
index $n$ in $\pi_1(\mathcal{G}_2)$ with $\Delta\cong
\pi_1(\mathcal{G}_2)$ and conjugacy classes of such subgroups. This
statements correspond to (\cite{We2}, Proposition 3).

\begin{proposition}\label{number of G2_in_G2}
The subgroups $\Delta$ of index $n$ in $\pi_1(\mathcal{G}_2)$
isomorphic to $\pi_1(\mathcal{G}_2)$ are in one-to-one
correspondence with the triples $(k,H,h)$, where
\begin{itemize}
\item $k$ is an odd positive divisor of $n$,
\item $H$ is a subgroup of index $\frac{n}{k}$ in $\ZZ^2$,
\item $h$ is a coset in $\ZZ^2/H$.
\end{itemize}
Consequently, the number of the above described subgroups is
$s_{\mathcal{G}_2,\mathcal{G}_2}(n)=\omega(n)-\omega\left(\frac{n}{2}\right)$.
\end{proposition}

\begin{proposition}\label{number of conjugacy_classes_G2_in_G2}
The conjugacy classes of subgroups $\Delta$ of index $n$ in
$\pi_1(\mathcal{G}_2)$ isomorphic to $\pi_1(\mathcal{G}_2)$ are in
one-to-one correspondence with the triples $(k,H,\bar{h})$, where
\begin{itemize}
\item $k$ is an odd positive divisor of $n$,
\item $H$ is a subgroup of index $\frac{n}{k}$ in $\ZZ^2$,
\item $\bar{h}$ is a coset in $\ZZ^2/\langle H, (2,0), (0,2) \rangle$.
\end{itemize}
Consequently, the number of conjugacy classes of the above described
subgroups is
$c_{\mathcal{G}_2,\mathcal{G}_2}(n)=\sigma_2(n)+2\sigma_2\left(\frac{n}{2}\right)-3\sigma_2\left(\frac{n}{4}\right)$.
\end{proposition}

\section{The structure of the groups $\pi_{1}(\mathcal{G}_{2})$ and $\pi_{1}(\mathcal{G}_{6})$ }
The groups $\pi_{1}(\mathcal{G}_{1})$, $\pi_{1}(\mathcal{G}_{2})$
and $\pi_{1}(\mathcal{G}_{6})$ is given by generators and relations
in the following way
\begin{equation}\label{all_relations}
\begin{aligned}
\pi_1(\mathcal{G}_1)&=\ZZ^3=\langle x,y,z:
xyx^{-1}y^{-1}=xzx^{-1}z^{-1}=yzy^{-1}z^{-1}=1\rangle,\\
\pi_1(\mathcal{G}_2)&=\langle x,y,z: xyx^{-1}y^{-1}=1,
x^z=x^{-1},y^z=y^{-1}\rangle,\\
\pi_{1}(\mathcal{G}_{6})&=\langle x, y, z:
xy^2x^{-1}y^2=yx^2y^{-1}x^2=xyz=1 \rangle .
\end{aligned}
 \end{equation}
See \cite{Wolf} or \cite{Conway}.

{\bf Remark.} The above representation of the group
$\pi_{1}(\mathcal{G}_{6})$ is indeed symmetric with respect to
permutations of $x$, $y$ and $z$. The relations
$xz^2x^{-1}z^2=yz^2y^{-1}z^2=zx^2z^{-1}x^2=zy^2z^{-1}y^2=1$ follow
from given above.

Next proposition provides the canonical form of an element in
$\pi_{1}(\mathcal{G}_{6})$.

\begin{proposition}\label{propG6-1}
\begin{itemize}
\item[(i)] Each element of $\pi_{1}(\mathcal{G}_{6})$ can be represented in the canonical form
$g_ix^{2a}y^{2b}z^{2c}$, where $g_i\in\{1,x,y,z\}$ and $a,b,c$ are
some integers.
\item[(ii)] The subgroup $\langle x^2,y^2,z^2\rangle$ is normal in
$\pi_{1}(\mathcal{G}_{6})$ and isomorphic to $\ZZ^3$.
\item[(iii)] The following relations holds:
\begin{equation}\label{multlawG6-1}
\begin{aligned}
x^{2a}y^{2b}z^{2c}\cdot x = x\cdot x^{2a}y^{-2b}z^{-2c},\\
x^{2a}y^{2b}z^{2c}\cdot y = y\cdot x^{-2a}y^{2b}z^{-2c},\\
x^{2a}y^{2b}z^{2c}\cdot z = z\cdot x^{-2a}y^{-2b}z^{2c}.
\end{aligned}
\end{equation}
\item[(iv)] The product $g_ig_j:\; g_i,g_j \in \{1,x,y,z\}$ is given
by Table~1. \hfil\hfil \linebreak
\begin{center}
\begin{tabular}{|l|l|l|l|l|}
\hline $g_ig_j$ & $g_j=1$ & $g_j=x$ & $g_j=y$ & $g_j=z$ \\
\hline $g_i=1$ & $1$ & $x$ & $y$ & $z$ \\
\hline $g_i=x$ & $x$ & $1\cdot x^2$ & $z\cdot z^{-2}$ & $y\cdot x^{-2}z^2$ \\
\hline $g_i=y$ & $y$ & $z\cdot x^2y^{-2}$ & $1\cdot y^2$ & $x\cdot x^{-2}$ \\
\hline $g_i=z$ & $z$ & $y\cdot y^{-2}$ & $x\cdot y^2z^{-2}$ & $1\cdot z^2$ \\
\hline
\end{tabular}
\linebreak\small{Table~1}
\end{center}
\item[(v)] The representation in the canonical form $w=g_ix^ay^bz^c$ for
each element $w\in \pi_{1}(\mathcal{G}_{6})$ is unique.
\end{itemize}
\end{proposition}

\begin{proof}
Items (i--iv) follow routinely from the representation
(\ref{all_relations}) of the group $\pi_{1}(\mathcal{G}_{6})$. To
prove (v) consider the set $G$ of all the expressions
$g_ix^{2a}y^{2b}z^{2c}$, where $g_i\in\{1,x,y,z\}$ and $a,b,c$ are
some integers. Define the multiplication by concatenation and
further reduction to the described above form through the relations
(iii) and (iv). Direct verification shows that $G$ is a group with
respect to this operation. Since the relations (iii) and (iv) are
derived from the relations of the group $\pi_{1}(\mathcal{G}_{6})$,
this group is a factor group of the group $G$. By the other hand,
one can verify that the relations of $\pi_{1}(\mathcal{G}_{6})$
holds in $G$, thus $G \cong \pi_{1}(\mathcal{G}_{6})$. In
particular, different canonical representations represent different
elements of $\pi_{1}(\mathcal{G}_{6})$.
\end{proof}

{\bf Notations.} Denote the subgroup $\langle x^2,y^2,z^2\rangle
\lhd \pi_{1}(\mathcal{G}_{6})$ by $\Lambda$. Also denote the natural
by homomorphism, of factorization $\pi_{1}(\mathcal{G}_{6}) \to
\pi_{1}(\mathcal{G}_{6})/\Lambda$ by $\phi$.

\begin{definition}
Let $g$ be an element of $\pi_{1}(\mathcal{G}_{6})$. In case
$g=x^{2a}y^{2b}z^{2c}$ we say that $g$ {\em has exponents} $2a$,
$2b$, $2c$ at $x$, $y$, $z$ respectively. In case
$g=zx^{2a}y^{2b}z^{2c}$ we say the respective exponents are $2a$,
$2b$, $2c+1$. Similarly in cases $g=x\cdot x^{2a}y^{2b}z^{2c}$ and
$g=yx^{2a}y^{2b}z^{2c}$. We denote the exponents of $g$ at $x$, $y$,
$z$ by $exp_x(g)$, $exp_y(g)$, $exp_z(g)$ respectively.
\end{definition}

We widely use the following statement, too trivial to be a lemma.
Let $g,h$ be some elements and $exp_y(g)$, $exp_z(g)$, $exp_y(h)$,
$exp_z(h)$ are even. Than $exp_x(gh)=exp_x(g)+exp_x(h)$.

Note that $\pi_{1}(\mathcal{G}_{6})/\Lambda \cong \ZZ_2^2$,
therefore there are only three possible isomorphism types of a
subgroup in $\ZZ_2^2$, it is either trivial, or $\ZZ_2$, or
$\ZZ_2^2$.

\begin{definition}\label{def_X_Delta}
Let $\Delta$ be a subgroup of finite index in
$\pi_{1}(\mathcal{G}_{6})$. In case $\phi(\Delta)=1$  by $X_\Delta,
Y_\Delta, Z_\Delta$ we refer to an arbitrary triple of generators of
$\Delta$. If $\phi(\Delta)=\{1,x\}$  by $Z_\Delta$ denote an
arbitrary element of $\Delta$ with the minimal positive odd exponent
at $x$, and by $X_\Delta, Y_\Delta$ denote an arbitrary pair of
generators of $\Delta \cap \langle x^2,y^2\rangle$. Similarly, in
case $\phi(\Delta)=\{1,y\}$ and $\phi(\Delta)=\{1,z\}$ ($Z_\Delta$
 denote an element with the minimal positive odd exponent at
$y$ and $z$ respectively). Finely, in case
$\phi(\Delta)=\{1,x,y,z\}$ by $X_\Delta, Y_\Delta, Z_\Delta$ denote
an arbitrary element with minimal positive odd exponent at $x$, $y$,
$z$ respectively.
\end{definition}

\begin{proposition}\label{G6_classification}
Let $\Delta$ be a subgroup of finite index in
$\pi_{1}(\mathcal{G}_{6})$. Then $\Delta$ has one of the following
three isomorphism types, defined by $|\phi(\Delta)|$. Subgroup
$\Delta$ is isomorphic to $\ZZ^3$, $\pi_{1}(\mathcal{G}_{2})$ and
$\pi_{1}(\mathcal{G}_{6})$ in case $|\phi(\Delta)|=1$,
$|\phi(\Delta)|=2$ and $|\phi(\Delta)|=4$ respectively. In all cases
$\Delta$ is generated by elements $X_\Delta, Y_\Delta, Z_\Delta$.
%
%
\end{proposition}

\begin{proof}
In case $|\phi(\Delta)|=1$ the triple $X_\Delta, Y_\Delta, Z_\Delta$
generates $\Delta$ by definition. Also, $\Delta$ is a subgroup of
$\Lambda \cong \ZZ^3$, thus $\Delta$ is a free abelian group. Since
$\Delta$ has finite index in $\Lambda$, we have $\Delta \cong
\ZZ^3$.

In case $|\phi(\Delta)|=2$ without loss of generality assume that
$\phi(\Delta)=\{\phi(1), \phi(x)\}$. Denote the exponent of
$Z_\Delta$ at $x$ by $m$. Then for each $g \in \Delta$ its exponent
at $x$ is a multiple of $m$. Otherwise multiplying either $g$ or
$g^{-1}$ by the convenient power of $Z_\Delta$ we get an element of
$\Delta$ with an odd exponent at $x$ strictly between $0$ and $m$,
which is the contradiction with the definition of $Z_\Delta$.

To prove that $X_\Delta, Y_\Delta, Z_\Delta$ generate $\Delta$
consider an arbitrary element $g \in \Delta$. Since its exponent at
$x$ is divisible by $m$, $gZ_\Delta^k \in \Delta \cap \langle y^2,
z^2\rangle $ for some $k$. Further, $\Delta \cap \langle y^2,
z^2\rangle$ is generated by $X_\Delta, Y_\Delta$ by virtue of their
definition. We claim that different expressions of the form
$X_\Delta^aY_\Delta^bZ_\Delta^c$ represent different elements $g
\in\Delta$. Indeed, the exponent of $g$ at $x$ uniquely determines
$c$, and $\Delta \cap \langle y^2, z^2\rangle \cong \ZZ^2$, so
different pairs $(a,b)$ provide different elements
$X_\Delta^aY_\Delta^b$.

Note that the elements $X_\Delta, Y_\Delta, Z_\Delta$ yield the
relations of the group $\pi_{1}(\mathcal{G}_{2})$ hold for $x,y,z$,
so we build the isomorphism $\pi_{1}(\mathcal{G}_{2}) \to \Delta$,
given by $x \mapsto X_\Delta$, $y \mapsto Y_\Delta$, $z \mapsto
Z_\Delta$.

In case $|\phi(\Delta)|=4$ we set:
$X_\Delta=x^my^{2r}z^{2s}$, $Y_\Delta=y^kx^{2t}z^{2u}$ and
$Z_\Delta=z^\ell x^{2v}y^{2w}$. Here $m, k, \ell, r,s,t,u,v,w$ are
integers, moreover $m, k, \ell$ are odd positives.

To prove that $\Delta$ is generated by $X_\Delta, Y_\Delta,
Z_\Delta$ do the following.

\begin{lemma} \label{auxiliary statement in prop 1} If for some element $g \in \Delta$ the numbers $exp_y(g)$
and $exp_z(g)$ are even, then $m \mid exp_x(g)$. \end{lemma}

The proof is similar to the case  $|\phi(\Delta)|=2$. Analogous
statements holds for any permutation of $x,y,z$.

Resume to the prove of \Cref{G6_classification}. Note that
$X_\Delta^2=x^{2m}$, $Y_\Delta^2=y^{2k}$ and $Z_\Delta^2=x^{2\ell}$.
By \Cref{auxiliary statement in prop 1} if for some $g \in \Delta$
all three numbers $exp_x(g)$, $exp_y(g)$, $exp_z(g)$ are even, then
they are divisible by $2m$, $2k$ and $2\ell$ respectively; thus $g$
can be expressed through $X_\Delta^2$, $Y_\Delta^2$, $Z_\Delta^2$.
If $g \in \Delta$ has one exponent odd, without loss of generality
$exp_x(g)$ is odd, then $gX_\Delta$ has all exponents even.

If $g \in \Delta$ has one exponent odd, without loss of generality
we assume that $exp_x(g)$ is odd, then $gX_\Delta$ has all exponents
even.

To prove the isomorphism part note that the element $X_\Delta
Y_\Delta Z_\Delta =x^{m-1-2t+2v}y^{1-k+2w-2r}z^{\ell-1+2u-2s}$ has
all three exponents even, thus  \Cref{auxiliary statement in prop 1}
implies $2m \mid m-1-2y+2v$, $2k \mid 1-k+2w-2r$ and $2\ell \mid
\ell-1+2u-2s$. So, by replacing $X_\Delta \mapsto X_\Delta
Y_\Delta^{2i}$ for some integer $i$ and doing similar replacements
for permuted generators, one can achieve that $X_\Delta Y_\Delta
Z_\Delta = 1$ and the property of $X_\Delta, Y_\Delta, Z_\Delta$
given in \Cref{def_X_Delta} holds.

Now note that the elements $X_\Delta, Y_\Delta, Z_\Delta$ yield
defining relations of the group $\pi_{1}(\mathcal{G}_{6})$. Then the
mapping $x \mapsto X_\Delta$, $y \mapsto Y_\Delta$, $z \mapsto
Z_\Delta$ spawns the epimorphism $\psi: \pi_{1}(\mathcal{G}_{6}) \to
\Delta$. We are going to prove that this epimorphism is indeed an
isomorphism.

Each element $g$ of $\Delta$ can be represented in the form
$g=g_iX_\Delta^{2a} Y_\Delta^{2b} Z_\Delta^{2c}$, where $g_i \in
\{1, X_\Delta, Y_\Delta, Z_\Delta\}$; whence such representation is
possible in $\pi_{1}(\mathcal{G}_{6})$.  So it is sufficient to
prove that the above representation  is unique for each $g \in
\Delta$. Assume the contrary, for some element $g$ there are two
different representations $g=g_iX_\Delta^{2a} Y_\Delta^{2b}
Z_\Delta^{2c}=g_i'X_\Delta^{2a'} Y_\Delta^{2b'} Z_\Delta^{2c'}$.

Note that for arbitrary $g \in \Delta$ hold $(x^2)^g=x^{\pm2}$,
$(y^2)^g=y^{\pm2}$ and $(z^2)^g=z^{\pm2}$, and the triple of signs
in the exponents is solely determined by $g_i$: $1 \mapsto (+,+,+)$,
$x \mapsto (+,-,-)$, $y \mapsto (-,+,-)$ and $z \mapsto (-,-,+)$.
Thus $g_iX_\Delta^{2a} Y_\Delta^{2b}
Z_\Delta^{2c}=g_i'X_\Delta^{2a'} Y_\Delta^{2b'} Z_\Delta^{2c'}$
implies $g_i=g_i'$. Then $X_\Delta^{2a} Y_\Delta^{2b}
Z_\Delta^{2c}=X_\Delta^{2a'} Y_\Delta^{2b'} Z_\Delta^{2c'}$, or
$x^{2m(a-a')}y^{2k(b-b')}z^{2\ell(c-c')}=1$, which is a
contradiction with \Cref{propG6-1} (iv).
\end{proof}

\section{Proof of \Cref{th-1-didicosm} and \Cref{th-2-didicosm}}

The isomorphism types of finite index subgroups are already provided
by \Cref{G6_classification}. So we will consider isomorphism types
separately in order to prove respective items of both theorems.

\subsection{Case $\Delta \cong \ZZ^3$}\label{enumeration_Z^3}
Recall that $\Lambda=\langle x^2,y^2,z^2\rangle$. By
\Cref{G6_classification} each subgroup $\Delta$ of index $n$ in
$\pi_{1}(\mathcal{G}_{6})$ with $\Delta \cong \ZZ^3$ has
$\phi(\Delta)=\{1\}$, that is $\Delta \leqslant \Lambda$. Since
$|\pi_{1}(\mathcal{G}_{6}):\Lambda|=4$, get
$|\Lambda:\Delta|=\frac{n}{4}$. Applying \Cref{number of
sublattices} one gets
$$
s_{\mathcal{{G}}_{1},\mathcal{{G}}_{6}}(n)=\omega\left(\frac{n}{4}\right).
$$

Now we proceed to enumeration of the conjugacy classes of subgroups.
Since $\Lambda$ is abelian, the group $\pi_{1}(\mathcal{G}_{6})$
acts by conjugation on subgroups of $\Lambda$ as
$\pi_{1}(\mathcal{G}_{6})/\Lambda\cong\ZZ_2^2$. Thus each conjugacy
class consists of one, two or four subgroups.

\begin{definition}\label{defG6-M}
By $\mathcal{M}_1$ denote the family of all normal subgroups
$\Delta$, by $\mathcal{M}_2$ and $\mathcal{M}_4$ denote the families
of subgroups $\Delta$, which belong to conjugacy classes, containing
two and four subgroups respectively. Also, by $\mathcal{M}_x$ denote
the family of subgroups $\Delta$ such that $\Delta^x=\Delta$.
$\mathcal{M}_y$ and $\mathcal{M}_z$ are defined similar way.
\end{definition}

Each $\Delta \in \mathcal{M}_1$ belongs to all three of
$\mathcal{M}_x$, $\mathcal{M}_y$, $\mathcal{M}_z$; while each
$\Delta \in \mathcal{M}_2$ belongs to exactly one of
$\mathcal{M}_x$, $\mathcal{M}_y$, $\mathcal{M}_z$. Needless to say
that $\Delta \in \mathcal{M}_4$ does not belongs to any of
$\mathcal{M}_x$, $\mathcal{M}_y$, $\mathcal{M}_z$. So
$3|\mathcal{M}_1|+|\mathcal{M}_2|=|\mathcal{M}_x|+|\mathcal{M}_y|+|\mathcal{M}_z|$.
Thus
\begin{equation}\label{eqG6-1}\aligned &
c_{\mathcal{{G}}_{1},\mathcal{{G}}_{6}}(n)=|\mathcal{M}_1|+\frac{|\mathcal{M}_2|}{2}+\frac{|\mathcal{M}_4|}{4}=
\frac{|\mathcal{M}_1|+|\mathcal{M}_2|+|\mathcal{M}_4|}{4}+\frac{3|\mathcal{M}_1|+|\mathcal{M}_2|}{4}
\\&
=\frac{s_{\mathcal{G}_{1},\mathcal{G}_{6}}(n)}{4}+\frac{|\mathcal{M}_x|+|\mathcal{M}_y|+|\mathcal{M}_z|}{4}.
\endaligned\end{equation}
Since in a suitable basis each of $Ad_x$, $Ad_y$, $Ad_z$ takes the
form $(a,b,c)\mapsto (-a,b,c)$, \Cref{number of mirror-preserved}
claims
$|\mathcal{M}_x|=|\mathcal{M}_y|=|\mathcal{M}_z|=\sigma_2\left(\frac{n}{4}\right)+3\sigma_2\left(\frac{n}{8}\right)$.
Thus
$$
c_{G_1,G_6}(n)=\frac{1}{4}\omega\left({\frac{n}{4}}\right)+\frac{3}{4}\sigma_2\left(\frac{n}{4}\right)+\frac{9}{4}\sigma_2\left(\frac{n}{8}\right).
$$
By definition, $|\mathcal{M}_1|$ is the number of normal subgroups
of index $n$ in $\pi_{1}(\mathcal{G}_{6})$ isomorphic to $\ZZ^3$. So
it is interesting in itself. It is explicitly calculated in
\cref{accessory remarks}.

\subsection{Case $\Delta \cong
\pi_{1}(\mathcal{G}_{2})$}\label{sectG2_in_G6} By
\Cref{G6_classification} each subgroup $\Delta$ of index $n$ in
$\pi_{1}(\mathcal{G}_{6})$ with $\Delta \cong
\pi_{1}(\mathcal{G}_{2})$ has $|\phi(\Delta)|=2$. In other words,
holds one of the inclusions $\Delta \leqslant \langle x, y^2,
z^2\rangle = \Gamma_x$, $\Delta \leqslant \langle y, x^2,
z^2\rangle=\Gamma_y$, $\Delta \leqslant \langle z, x^2,
y^2\rangle=\Gamma_z$. Since the above groups have index 2 in
$\pi_{1}(\mathcal{G}_{6})$, subgroup $\Delta$ has index
$\frac{n}{2}$ in the respective subgroup. Further, $\Delta$ belongs
to just one of $\Gamma_x, \Gamma_y, \Gamma_z$ hence the intersection
of each two of them is the abelian group $\Lambda=\langle
x^2,y^2,x^2 \rangle$.

%

Also, subgroups $\Gamma_x, \Gamma_y, \Gamma_z$ are permutable by
some outer automorphism of $\pi_{1}(\mathcal{G}_{6})$, which
permutes $x,y,z$. Thus it is sufficient to enumerate the subgroups
of $\Gamma=\Gamma_x$. Further during this subsection $\Delta$
denotes a subgroup of index $\frac{n}{2}$ in $\Gamma$ isomorphic to
$\pi_{1}(\mathcal{G}_{2})$. Pay attention, in spite of
$\Delta\cong\Gamma $, $\Delta$ is a non-trivial subgroup in
$\Gamma$.

Since $\Gamma \cong \pi_{1}(\mathcal{G}_{2})$, the number of the
above subgroups $\Delta$ in $\Gamma$ is provided by \Cref{number of
G2_in_G2}, thus
$$
s_{\mathcal{G}_2,\mathcal{G}_6}(n)=3s_{\mathcal{G}_2,\mathcal{G}_2}\left(\frac{n}{2}\right)=3\omega\left(\frac{n}{2}\right)-3\omega\left(\frac{n}{4}\right).
$$

To enumerate conjugacy classes we need one more definition.

\begin{definition}
Consider a subgroup $\Delta$.  The set of subgroups
$\{\Delta^\gamma| \gamma\in\Gamma\}$ we call a {\em partial
conjugacy class} $\Delta^\Gamma$.
\end{definition}

An enumeration of partial conjugacy classes of subgroups $\Delta$ is
given by \Cref{number of conjugacy_classes_G2_in_G2}. To enumerate
conjugacy classes note that $\pi_{1}(\mathcal{G}_{6})=\Gamma\cup
y\Gamma$, so $\Gamma$ has index 2 in $\pi_{1}(\mathcal{G}_{6})$.
Consequently $\Gamma$ is normal in $\pi_{1}(\mathcal{G}_{6})$. Thus
for each $\Delta$ its conjugacy class consists of one or two partial
conjugacy classes $\Delta^\Gamma$ and $(\Delta^\Gamma)^y$ depending
upon whether partial conjugacy classes $\Delta^\Gamma$ and
$(\Delta^\Gamma)^y$ coincide or not.

{\bf Notation.} By $\mathcal{K}_1$ denote the set of partial
conjugacy classes $\Delta^\Gamma$, such that the equality
$\Delta^\Gamma=(\Delta^\Gamma)^y$ holds. By $\mathcal{K}_2$ denote
the set of partial conjugacy classes $\Delta^\Gamma$ with
$\Delta^\Gamma\neq(\Delta^\Gamma)^y$.

In the introduced notation
\begin{equation}\label{locally_needed_equation} c_{\mathcal{G}_2,\mathcal{G}_6}(n)=3\left(|\mathcal{K}_1|+\frac{|\mathcal{K}_2|}{2}\right)=3\left(\frac{|\mathcal{K}_1|+|\mathcal{K}_2|}{2}+\frac{|\mathcal{K}_1|}{2}\right).
\end{equation} \Cref{number of conjugacy_classes_G2_in_G2} implies
$|\mathcal{K}_1|+|\mathcal{K}_2|=\sigma_2\left(\frac{n}{2}\right)+2\sigma_2\left(\frac{n}{4}\right)-3\sigma_2\left(\frac{n}{8}\right)$.
All that's left is to calculate $|\mathcal{K}_1|$, this is done in
\Cref{invariant_partial_G2_classes}. First we need the following
auxiliary statement.

\begin{lemma}\label{inclusion_exclusion_for_d_3} The following identity holds
$$
d_3(n)-3d_3\left(\frac{n}{2}\right)+3d_3\left(\frac{n}{4}\right)-d_3\left(\frac{n}{8}\right)=\begin{cases}
d_3(n) &\text{if $n$ is odd}\\
0 &\text{if $n$ is even}
\end{cases}.
$$
\end{lemma}

\begin{proof}
Consider all factorizations $n=abc$ and use inclusion-exclusion
formula for all triples of parities of $a,b,c$.
\end{proof}

\begin{proof}
Recall that by definition $d_3(n)$ is the number of ordered positive
integer factorizations $abc=n$. Then in case of an odd $n$ equality
$d_3(n)-3d_3\left(\frac{n}{2}\right)+3d_3\left(\frac{n}{4}\right)-d_3\left(\frac{n}{8}\right)=d_3(n)$
holds because terms $d_3\left(\frac{n}{2}\right),
d_3\left(\frac{n}{4}\right), d_3\left(\frac{n}{8}\right)$ vanish.

Assume $n$ is even. Note that positive integer factorizations
$abc=n$ with even $a$ are enumerated by
$d_3\left(\frac{n}{2}\right)$. Indeed, they bijectively correspond
to factorizations $\frac{a}{2}bc=\frac{n}{2}$. Same holds for
factorizations $abc=n$ with even $b$, and factorizations $abc=n$
with even $c$. Similarly, factorizations $abc=n$ with even $a$ and
$b$ simultaneously are enumerated by $d_3\left(\frac{n}{4}\right)$.
The same holds for permuted $a,b,c$. Finally, factorizations $abc=n$
with even $a,b,c$ are enumerated by $d_3\left(\frac{n}{8}\right)$.
Applying inclusion-exclusion formula we get that
$d_3(n)-3d_3\left(\frac{n}{2}\right)+3d_3\left(\frac{n}{4}\right)-d_3\left(\frac{n}{8}\right)$
enumerates factorizations $abc=n$, where all three $a,b,c$ are odd.
Since $n$ is even such factorization is impossible, the above
expression vanishes.
\end{proof}

Next lemma finally calculates $|\mathcal{K}_1|$.
\begin{lemma}\label{invariant_partial_G2_classes}
$$
|\mathcal{K}_1|=d_3(n/2)-d_3(n/4)-3d_3(n/8)+5d_3(n/16)-2d_3(n/32).
$$
\end{lemma}

{\bf Remark.} The above formula for $|\mathcal{K}_1|$ looks
horribly; actually it means the following. Let $n=2^qr$ where
$2\nmid s$. Then
\begin{itemize}
\item if $q=0$ then $|\mathcal{K}_1|=0$;
\item if $q=1$ then $|\mathcal{K}_1|=d_3(r)$;
\item if $q=2$ then $|\mathcal{K}_1|=2d_3(r)$;
\item if $q>2$ then $|\mathcal{K}_1|=0$.
\end{itemize}

\begin{proof}
Let $\Delta$ be a subgroup of even index $n$ in
$\pi_{1}(\mathcal{G}_{6})$, such that $\Delta \cong
\pi_{1}(\mathcal{G}_{2})$ and $(\Delta^\Lambda)^y=\Delta^\Lambda$.

 Let
$Z_\Delta=x^ky^{2s}z^{2t}$, where $k$ is an odd positive, and
 $k \mid \frac{n}{2}$. We set
$H_\Delta=\Delta\cap\langle y^2,z^2\rangle$. Further we identify
$\langle y^2,z^2\rangle$ with $\ZZ^2$, that is we address an element
$y^{2a}z^{2b}$ as $(a,b)$.

 \Cref{number of
conjugacy_classes_G2_in_G2} implies that the condition
$(\Delta^\Lambda)^y=\Delta^\Lambda$ means that the groups $\Delta$
and $\Delta^y$ have the same triples $(k,H,\bar{h})$ of invariants.
Consider two conditions:  (i) the groups $\Delta$ and $\Delta^y$
share the same invariant $H$, (ii)  the groups $\Delta$ and
$\Delta^y$ share the same invariant $\bar{h}$.

Condition (i) means that $\Delta^y\cap \langle y^2,z^2\rangle
=\Delta\cap \langle y^2,z^2\rangle$, i.e. $Ad_y(H)=H$. Since the
action of $Ad_y$ on $\langle y^2,z^2\rangle$ is given by
$(u,v)\mapsto (u,-v)$, \Cref{number of mirror-preserved_in_Z^2}
claims that either $H=\langle (a,0),(0,b)\rangle$ or $H=\langle
(a,0),(a/2,b)\rangle$, where $a,b>0$ and $ab=\frac{n}{2k}$;
additionally $a$ is even in the second case. We say that a subgroups
$H$ is {\em of the first type}  if $H=\langle (a,0),(0,b)\rangle$,
likewise $H$ is {\em of the second type}  if $H=\langle
(a,0),(a/2,b)\rangle$.

Consider condition (ii). Note that
$(Z_\Delta)^y=(x^ky^{2s}z^{2t})^y=x^{-k}y^{2s-2}z^{-2t-2}$. Thus the
element $(Z_\Delta^y)^{-1} \in \Delta^y$ satisfies the definition of
the element $Z_{\Delta^y}$. So the corresponding value of $\bar{h}$
is $(s-1,-t-1)$. Then the condition (ii) is reformulated as
$(s,t)\in (s-1,-t-1)+\langle H,(2,0),(0,2)\rangle$, or equivalently
$(1,1) \in \langle H,(2,0),(0,2)\rangle$.

In case  $H=\langle (a,0),(0,b)\rangle$ this implies $a$ and $b$ are
odd, thus $\frac{n}{2}=kab$ is odd ($k$ is odd due to \Cref{number
of G2_in_G2}). Vice versa, in case $\frac{n}{2}$ is odd, an
arbitrary positive factorization $\frac{n}{2}=kab$ spawns the unique
group $H=\langle (a,0),(0,b)\rangle$, that defines the unique coset
$\ZZ^2/\langle H,(2,0),(0,2)\rangle$, so we get the unique partial
conjugacy class $\Delta^\Lambda$ with
$\Delta^\Lambda=(\Delta^\Lambda)^y$ corresponding to each
factorization $\frac{n}{2}=kab$. Thus there are $d_3(\frac{n}{2})$
conjugacy classes of the first type if $\frac{n}{2}$ is odd. Then by
\Cref{inclusion_exclusion_for_d_3} the first type provides
$d_3(\frac{n}{2})-3d_3(\frac{n}{4})+3d_3(\frac{n}{8})-d_3(\frac{n}{16})$
partial conjugacy classes in $\mathcal{K}_1$.

In case  $H=\langle (a,0),(a/2,b)\rangle$ condition $(1,1) \in
\langle H,(2,0),(0,2)\rangle$ implies $b$ and $\frac{a}{2}$ are odd.
Then $|\ZZ^2/\langle (a,0),(a/2,b),(2,0),(0,2)\rangle|=2$. Vice
versa, if $\frac{n}{2}$ is even but not divisible by 4, then each
factorization $\frac{n}{2}=kab$, where $k,b$ are odd provides the
unique subgroup $H$ of the second type. In turn, each subgroup $H$
provides two cosets $\bar{h}$ because $|\ZZ^2/\langle
H,(2,0),(0,2)\rangle|=2$. Thus if $\frac{n}{4}$ is odd there are
$2d_3(\frac{n}{4})$ partial conjugacy classes of the second type in
$\mathcal{K}_1$, again by \Cref{inclusion_exclusion_for_d_3} this
amount is equal to
$2d_3\left(\frac{n}{4}\right)-6d_3\left(\frac{n}{8}\right)+6d_3\left(\frac{n}{16}\right)-2d_3\left(\frac{n}{32}\right)$.

Summing up one gets
$|\mathcal{K}_1|=d_3(n/2)-d_3(n/4)-3d_3(n/8)+5d_3(n/16)-2d_3(n/32)$.
\end{proof}

Substituting the result of \Cref{invariant_partial_G2_classes} into
equation (\ref{locally_needed_equation}) we get
\begin{align*}
c_{\mathcal{G}_2,\mathcal{G}_6}(n)&=3\left(\frac{|\mathcal{K}_1|+|\mathcal{K}_2|}{2}+\frac{|\mathcal{K}_1|}{2}\right)\\&=\frac{3}{2}\left(\sigma_2\!\!\left(\frac{n}{2}\right)+2\sigma_2\!\!\left(\frac{n}{4}\right)-3\sigma_2\!\!\left(\frac{n}{8}\right)+d_3\!\!\left(\frac{n}{2}\right)-d_3\!\!\left(\frac{n}{4}\right)-3d_3\!\!\left(\frac{n}{8}\right)+5d_3\!\!\left(\frac{n}{16}\right)-2d_3\!\!\left(\frac{n}{32}\right)\right).
\end{align*}

\subsection{Case $\Delta \cong \pi_{1}(\mathcal{G}_{6})$}
We claim that the following two propositions holds.

{\bf Notation.} Given integers $m,n$ with $n > 0$, by $[m]_n$ denote
the integer number, defined by $0 \le [m]_n < n$ and $m\equiv[m]_n
\mod n$.

\begin{proposition}\label{enumeration_of_G6_in_G6}
The subgroups $\Delta$ of index $n$ in $\pi_1(\mathcal{G}_6)$
isomorphic to $\pi_1(\mathcal{G}_6)$ are in one-to-one
correspondence with the $6$-plets $(k,\ell,m,u,v,w)$, $0 \le v < m$,
$0 \le u < \ell$, $0\le w < k$ where $k,\ell,m $ are odd positives
and $k\ell m=n$. Moreover, a subgroup $\Delta$ is generated by
elements $X_\Delta=x^my^{[1-k+2w]_{2k}}z^{[\ell-1+2u]_{2\ell}}$,
$Y_\Delta=y^kx^{[m-1+2v]_{2m}}z^{2u}$ and $Z_\Delta=z^\ell
x^{2v}y^{2w}$.
\end{proposition}

\begin{proof}
The correspondence from the set of subgroups $\Delta$ onto the set
of $6$-plets of the above form is built in the proof of
\Cref{G6_classification}.

To build the back correspondence consider the subgroup generated by
$X_\Delta, Y_\Delta, Z_\Delta$. Consider the following set
$$\langle x^{2m}, y^{2k}, z^{2\ell}\rangle \cup
X_\Delta\langle x^{2m}, y^{2k}, z^{2\ell}\rangle \cup
Y_\Delta\langle x^{2m}, y^{2k}, z^{2\ell}\rangle \cup
Z_\Delta\langle x^{2m}, y^{2k}, z^{2\ell}\rangle.$$ Direct
verification shows that it is the subgroup. That is, the group
$\langle X_\Delta, Y_\Delta, Z_\Delta \rangle$ is a subgroup of
index $k\ell m=n$ in $\pi_1(\mathcal{G}_6)$, isomorphic to
$\pi_1(\mathcal{G}_6)$ itself, by virtue of
\Cref{G6_classification}.
\end{proof}

\begin{proposition}\label{enumeration_of_G6_in_G6_conjugacy_classes}
The conjugacy classes of subgroups $\Delta$ of index $n$ in
$\pi_1(\mathcal{G}_6)$ isomorphic to $\pi_1(\mathcal{G}_6)$ are in
one-to-one correspondence with the triples $(k,\ell,m)$, where
$k,\ell,m $ are odd positive integers and $k\ell m=n$.
\end{proposition}

\begin{proof}
By \Cref{enumeration_of_G6_in_G6}, a subgroup $\Delta$ of the above
type is defined by the $6$-plet $(k,\ell,m,u,v,w)$. Note that the
conjugation with the elements $x^2$, $y^2$, $z^2$ acts on the above
$6$-plets in the following way:
\begin{align*}
Ad_{x^2}:\, (k,\ell,m,u,v,w)&\mapsto (k,\ell,m,u,[v-2]_{m},w),\\
Ad_{y^2}:\, (k,\ell,m,u,v,w)&\mapsto (k,\ell,m,u,v,[w-2]_{k}),\\
Ad_{z^2}:\, (k,\ell,m,u,v,w)&\mapsto (k,\ell,m,[u-2]_{\ell},v,w).
\end{align*}
Thus each two subgroups with the same triple $(k,\ell,m)$ are
conjugated by a suitable element of $\langle x^2,y^2,z^2\rangle$.
Obviously the conjugation with any element can not change the triple
$(k,\ell,m)$.
\end{proof}

\begin{corollary}
$$
s_{\mathcal{G}_6,\mathcal{G}_6}(n)=n\left(d_3(n)-3d_3\left(\frac{n}{2}\right)+3d_3\left(\frac{n}{4}\right)-d_3\left(\frac{n}{8}\right)\right),
$$
$$
c_{\mathcal{G}_6,\mathcal{G}_6}(n)=d_3(n)-3d_3\left(\frac{n}{2}\right)+3d_3\left(\frac{n}{4}\right)-d_3\left(\frac{n}{8}\right).
$$
\end{corollary}
\begin{proof}
To get the second formula use \Cref{inclusion_exclusion_for_d_3}. To
proceed to the first formula note that for any triple $(k,\ell,m)$
there are $m$ choices of $v$, $k$ choices of $w$ and $\ell$ choices
of $u$. By the second formula there exist
$d_3(n)-3d_3\left(\frac{n}{2}\right)+3d_3\left(\frac{n}{4}\right)-d_3\left(\frac{n}{8}\right)$
triples $(k,\ell,m)$, each corresponds to exactly $n$ different
$6$-plets $(k,\ell,m,u,v,w)$.
\end{proof}


\section{Additional Notes}\label{accessory remarks}
The purpose of this section is to enumerate the normal subgroups
$\Delta$ of index $n$ in $\pi_{1}(\mathcal{G}_{6})$, such that
$\Delta\cong \ZZ^3$. In the notations of \Cref{enumeration_Z^3} the
following holds.
\begin{proposition}\label{propG6-normalZ^3}
The number of normal subgroups of index $n$ in
$\pi_{1}(\mathcal{G}_{6})$, isomorphic to $\ZZ^3$ is given by the
formula
\begin{equation*}\aligned &
|\mathcal{M}_1|= d_3(n/4) + 4 d_3(n/8) + d_3(n/16) + 2 d_3(n/32).
\endaligned\end{equation*}
\end{proposition}

\begin{proof}
As it was shown above, a subgroup $\Delta$ of described type is a
subgroup of index $\frac{n}{4}$ in $\Lambda$, so we use \Cref{number
of sublattices}. The matrix $\bigl(\begin{smallmatrix}a & d & f \\
0 & b & e \\ 0 & 0 & c
\end{smallmatrix}\bigr)$ determines a
normal subgroup if and only if $(2d,0,0)\in \langle(a,0,0) \rangle$
and $(2f,2e,0)\in \langle(a,0,0),(d,b,0) \rangle$. Thus we have to
find the number of integer matrixes among the following eight: $\bigl(\begin{smallmatrix}  a & 0 & 0 \\
0 & b & 0 \\ 0 & 0 & c\end{smallmatrix}\bigr)$, $\bigl(\begin{smallmatrix} a & 0 & a/2 \\
0 & b & 0 \\ 0 & 0 & c \end{smallmatrix}\bigr)$, $\bigl(\begin{smallmatrix} a & a/2 & 0 \\
0 & b & 0 \\ 0 & 0 & c \end{smallmatrix}\bigr)$, $\bigl(\begin{smallmatrix} a & a/2 & a/2 \\
0 & b & 0 \\ 0 & 0 & c \end{smallmatrix}\bigr)$, $\bigl(\begin{smallmatrix} a & 0 & 0 \\
0 & b & b/2 \\ 0 & 0 & c \end{smallmatrix}\bigr)$, $\bigl(\begin{smallmatrix} a & 0 & a/2 \\
0 & b & b/2 \\ 0 & 0 & c \end{smallmatrix}\bigr)$, $\bigl(\begin{smallmatrix} a & a/2 & a/4 \\
0 & b & b/2 \\ 0 & 0 & c \end{smallmatrix}\bigr)$,
$\bigl(\begin{smallmatrix} a & a/2 & 3a/4 \\
0 & b & b/2 \\ 0 & 0 & c \end{smallmatrix}\bigr)$. The first matrix
is always integer, that is appears $d_3(\frac{n}{4})$ times, once in
each factorization of the type $abc=\frac{n}{4}$. The next three
matrices are integer if $a$ is even, that is they appear in
$d_3\left(\frac{n}{8}\right)$ factorizations $2abc=\frac{n}{4}$.
Analogously the fifth matrix is integer if $b$ is even, so this
matrix is counted $d_3\left(\frac{n}{8}\right)$ times. The sixth
matrix is integer if $a$ and $b$ are both even, it is counted
$d_3\left(\frac{n}{16}\right)$ times. The seventh and eighth
matrices are integer if $4\mid a$ and $b$ is even, they are counted
$d_3\left(\frac{n}{32}\right)$ times. So
$$
|\mathcal{M}_1|= d_3(n/4) + 4 d_3(n/8) + d_3(n/16) + 2 d_3(n/32).
$$\end{proof}

{\bf Remark.} The respective Dirichlet generating function is
$2^{-2s}\big(1+4\cdot2^{-s}+2^{-2s}+2\cdot2^{-3s}\big)\zeta^3(s)$,
see Appendix for details.

\begin{proposition}\label{propG6-normalG_2}
The number of normal subgroups $\Delta$ of index $n$ in
$\pi_{1}(\mathcal{G}_{6})$ isomorphic to $\pi_{1}(\mathcal{G}_{2})$
equals to $3$ in case $n$ is of the form $4m+2$; $6$ in case $n$ is
of the form $8m+4$; $0$ in all other cases.
\end{proposition}

\begin{proof}
In this proof we follow notations and overall ideas of
\Cref{sectG2_in_G6} (see first two paragraphs). So it is sufficient
to enumerate the subgroups $\Delta$ in $\Gamma_x$ of the type
considered in \Cref{propG6-normalG_2}. \Cref{number of G2_in_G2}
claims that a subgroup $\Delta$ of the above type is uniquely
defined by a triple $(k,H,h)$, while \Cref{number of
conjugacy_classes_G2_in_G2} describes the transformations of such
triple under conjugation of the group $\Delta$ with an element $g
\in \Gamma_x$. Similarly, the proof of
\Cref{invariant_partial_G2_classes} describes the transformation of
triple $(k,H,h)$ induced by the conjugation of the group $\Delta$ by
an element $y$. Since $\pi_{1}(\mathcal{G}_{6})=\Gamma_x \cup y
\Gamma_x$, each conjugation can be achieved as a composition of
described above.

Summarizing, we get that the invariant $h$ is preserved by any
conjugation if and only if \linebreak$(2,0), (0,2), (1,1) \in H$.
This holds for just two subgroups $H$ having index $1$ and $2$ in
$\langle y^2, z^2\rangle \cong \ZZ^2$ respectively. Both subgroups
are normal in $\mathcal{G}_{6}$, thus in both cases $H$ is
automatically preserved by any conjugation. So, case
$k=\frac{n}{2}$, where $k$ is odd, provides one subgroup of
$\Gamma_x$ which is normal in $\pi_{1}(\mathcal{G}_{6})$. Case
$k=\frac{n}{4}$, where $k$ is odd, provides two subgroups of
$\Gamma_x$ which are normal in $\pi_{1}(\mathcal{G}_{6})$. No other
values of $k$ provides normal subgroups.
\end{proof}

\begin{proposition}
Any normal subgroup $\Delta \unlhd  \pi_{1}(\mathcal{G}_{6})$
isomorphic to  $\pi_{1}(\mathcal{G}_{6})$ coincide with whole
$\pi_{1}(\mathcal{G}_{6})$.
\end{proposition}

Actually this is shown in the proof of
\Cref{enumeration_of_G6_in_G6_conjugacy_classes}.


\section*{Appendix\label{Appendix}}
Given a sequence $\{f(n)\}_{n=1}^\infty$, the formal power series
$$
\widehat{f}(s)=\sum_{n=1}^\infty\frac{f(n)}{n^s}
$$
is called a Dirichlet generating function for
$\{f(n)\}_{n=1}^\infty$. To reconstruct the sequence $f(n)$ from
$\widehat{f}(s)$ one can use Perron's formula (\cite{Apostol}, Th.
11.17). Given sequences $f(n)$ and $g(n)$ we call their {\em
convolution} $(f\ast g)(n) = \sum_{k \mid n}f(k)g(\frac{n}{k})$. In
terms of Dirichlet generating series the convolution of sequences
corresponds to the multiplication of generating series
$\widehat{f\ast g}(s)=\widehat{f}(s)\widehat{g}(s)$. For the above
facts see, for example, (\cite{Apostol}, Ch. 11--12).

Here we present the Dirichlet generating functions for the sequences
$s_{H,G}(n)$ and $c_{H,G}(n)$. Since Theorems 1 and 2 provide the
explicit formulas, the remainder is done by direct calculations.

Consider the Riemann zeta function
$\displaystyle\zeta(s)=\sum_{n=1}^{\infty}\frac{1}{n^s}$. Following
\cite{Apostol} note that
\begin{align*}
\widehat{\sigma}_0(s)&=\zeta^2(s), &
\widehat{\sigma}_2(s)&=\zeta^2(s)\zeta(s-1), &
\widehat{d}_3(s)&=\zeta^3(s), & \widehat{\omega}(s)&=
\zeta(s)\zeta(s-1)\zeta(s-2).
\end{align*}

\begin{center}
\small{Table~2. Dirichlet generating functions for the sequences
$s_{H,\mathcal{G}_6}(n)$ and $c_{H,\mathcal{G}_6}(n)$.}
\end{center}
\begin{tabular}{|c|c|p{8.5cm}|}\hline
$H$&$s_{H,\mathcal{G}_6}$&$c_{H,\mathcal{G}_6}$\\\hline
$\pi_1(\mathcal{G}_1)$&$4^{-s}\zeta(s)\zeta(s-1)\zeta(s-2)$&
$4^{-s-1}\zeta(s)\zeta(s-1)\big(\zeta(s-2)+3(1+3\cdot2^{-s})\zeta(s)\big)$
\\\hline $\pi_1(\mathcal{G}_2)$&$2^{-s}\big(1-2^{-s}\big)\zeta(s)\zeta(s-1)\zeta(s-2)$&$3\cdot2^{-s-1}(1-2^{-s})\zeta^2(s)\big((1+3\cdot2^{-s})\zeta(s-1)+{(1-2^{-s})^2(1+2^{-s+1})\zeta(s)}\big) $\\\hline
$\pi_1(\mathcal{G}_1)$&$\big(1-2^{-s+1}\big)^3\zeta^3(s-1)$&$\big(1-2^{-s}\big)^3\zeta^3(s)
$
\\\hline
\end{tabular}\medskip

\section*{Acknowledgements}
The authors are grateful to an anonymous referee for pointing out a
deeper background to the subject of our research.


\bigskip

\address{G.~Chelnokov \\
 National Research University Higher School of Economics, Moscow, Russia
 \\}
 {\small\tt grishabenruven@yandex.ru }


\address{A.~Mednykh \\
 Sobolev Institute of Mathematics, Novosibirsk, Russia \\
 Novosibirsk State University, Novosibirsk, Russia\\}
{ mednykh@math.nsc.ru }

\end{document}